\documentclass[11pt,letterpaper]{amsart}

\usepackage{tikz}
\tikzstyle{every node}=[circle, draw, fill=black!50,
                        inner sep=0pt, minimum width=3pt]

\usepackage{amsfonts,amsmath,latexsym,color,epsfig,hyperref,enumitem, amssymb,bbm}

\newtheorem{theorem}{Theorem}%[section]
\newtheorem{lemma}{Lemma}%[section]

\newtheorem{prop}[lemma]{Proposition}

\renewcommand{\le}{\leqslant}
\renewcommand{\ge}{\geqslant}

\def\qed{\ifvmode\mbox{ }\else\unskip\fi\hskip 1em plus 10fill$\Box$}

\input{epsf}
\makeatletter
\def\Ddots{\mathinner{\mkern1mu\raise\p@
\vbox{\kern7\p@\hbox{.}}\mkern2mu
\raise4\p@\hbox{.}\mkern2mu\raise7\p@\hbox{.}\mkern1mu}}
\makeatother

\def\Z{\mathbb Z}

\def\N{\mathbb N}
\def\E{\mathbb E}

\title{ Most integers are not a sum of two palindromes}
\author{Dmitrii Zakharov}
\thanks{Zakharov's research was supported by the Jane Street Graduate Fellowship.}
\address{Department of Mathematics, Massachusetts Institute of Technology, Cambridge, MA 02139, USA}
\email{zakhdm@mit.edu}

\date{}

\begin{document}

\begin{abstract}
For $g \ge 2$, we show that the number of positive integers at most $X$ which can be written as sum of two base $g$ palindromes is at most $\frac{X}{\log^c X}$. This answers a question of Baxter, Cilleruelo and Luca.
\end{abstract}

\maketitle

Fix an integer $g \ge 2$. Every positive integer $a \in \N$ has a base $g$ representation, i.e. it can be uniquely written as
\begin{equation}\label{eq1}
a = \overline{a_n a_{n-1} \ldots a_0} = \sum_{i=0}^n g^i a_i, ~ \text{ where }a_i \in \{0, 1, \ldots, g-1\} \text{ and }a_n\neq 0.    
\end{equation}
A number $a \in \N$ with representation (\ref{eq1}) is called {\em a base $g$ palindrome} if $a_i = a_{n-i}$ holds for all $i=0, \ldots, n$. Baxter, Cilleruelo, and Luca \cite{Cilleruelo} studied additive properties of the set of base $g$ palindromes. Improving on a result of Banks \cite{Banks}, they showed that every positive integer can be written as a sum of three palindromes (provided that $g\ge 5$). They also showed that the number of integers at most $X$ which are sums of two palindromes is at least $X e^{-c_1\sqrt{\log X}}$ and at most $c_2 X$, for some constants $c_1 > 0$ and $c_2 < 1$ depending on $g$, and asked whether a positive fraction of integers can be written as a sum of two base $g$ palindromes. This was later reiterated by Green in his list of open problems as Problem 95. We answer this question negatively:

\begin{theorem}\label{main}
    For any integer $g \ge 2$ there exists a constant $c > 0$ such that
    $$
    \# \{ n < X:~ n \text{ is a sum of two base }g\text{ palindromes} \} \le \frac{X}{\log^{c} X},
    $$
    for all large enough $X$.
\end{theorem}

It is an interesting open problem to close the gap between this result and the lower bound of Baxter, Cilleruelo and Luca \cite{Cilleruelo}. We now proceed to the proof.

For $n \ge 1$, let $P_n$ be the set of base $g$ palindromes with exactly $n$ digits and $P = \bigcup_{n\ge 1}P_n$ be the set of all base $g$ palindromes. Note that 
$$
|P_n| = \begin{cases}
    g^{n/2}-g^{n/2-1}, ~ n\text{ is even,}\\
    g^{(n+1)/2}-g^{(n-1)/2}, ~n\text{ is odd.}
\end{cases}
$$
For an integer $N \ge 1$, we write $[N] = \{0, 1, \ldots, N-1\}$. For $A, B \subset \Z$ we let $A+B = \{a+b, ~a\in A, ~b \in B\}$ denote the sumset of $A$ and $B$.
Let $k \ge 1$ be sufficiently large and let $X = g^k$, it is enough to consider numbers $X$ of this form only.  With this notation, our goal is to upper bound the size of the intersection $(P + P) \cap [X]$.
We have 
$$
(P + P) \cap [X] = \bigcup_{k \ge n \ge m \ge 1} (P_n + P_m) \cap [X]
$$
and so we can estimate
\begin{equation}\label{eqsum}
|(P+P) \cap [X]| \le \sum_{k \ge n \ge m \ge 1} |P_n + P_m|.    
\end{equation}
We have $|P_n| \le g^{(n+1)/2}$, $|P_{m}| \le g^{ (m+1)/2}$ so using the trivial bound $|P_n+P_m| \le |P_n| |P_m|$ we can immediately get rid of the terms where $m$ is small:
\begin{equation}\label{eqcrap}
\sum_{\substack{k \ge n \ge m \ge 1\\ m \le k-\gamma \log k}} |P_n + P_m| \le \sum_{k \ge n \ge 1} |P_n| \cdot \sum_{m \le k-\gamma \log k} |P_m| \le     
\end{equation}
$$
\le  \sum_{k \ge n \ge 1} |P_n| \cdot 4 g^{(k+1)/2 -\gamma \log k/2} \le 16 g^{k+1 - \gamma \log k /2} \lesssim \frac{X}{k^{\gamma \log g/2}} \sim \frac{X}{(\log X)^{\gamma \log g/2}},
$$
where $\gamma > 0$ is a small constant which we will choose. 
Now we focus on a particular sumset $P_n + P_m$ from the remaining range. Write $m = n-d$ for some $d \ge 0$. 

For an integer $a = \overline{a_n \ldots a_0} $ let $r(a) = \overline{a_0 \ldots a_n}$ be the integer with the reversed digit order in base $g$ (we allow some leading zeros here). For $d\ge 0$ define
$$
a = \overline{1 \underbrace{0\ldots 0}_{d}1 }, ~ b = \overline{0 \underbrace{0\ldots 0}_{d}0}, ~
a' = \overline{0 \underbrace{\ell \ldots \ell}_{d} 0}, ~ b' = \overline{\underbrace{0 \ldots 0}_{d} 11},
$$
where we denoted $\ell = g-1$. These strings are designed to satisfy the following:
\begin{equation}\label{eq2}
a + b = a' + b' ~\text{ and }~ g^{d} r(a) + r(b) = g^{d} r(a') + r(b').    
\end{equation}
Indeed, note that 
$$
a' = \sum_{i=1}^d g^{i}\ell = g^{d+1}-g = (g^{d+1}+1) + 0 - (g+1) = a+b-b'
$$
and
$$
g^d r(a') = g^d a' = g^{2d+1}-g^{d+1} = g^{d}(g^{d+1}+1) + 0 - (g^{d+1}+g^d) = g^d r(a) + r(b) - r(b').
$$
We claim that the fact that (\ref{eq2}) holds for some $a, b, a', b'$ forces the sumset $P_n + P_{n-d}$ to be small. Roughly speaking, whenever palindromes $p \in P_n$ and $q \in P_{n-d}$ contain strings $a$ and $b$ on the corresponding positions, we can swap $a$ with $a'$ and $b$ with $b'$ to obtain a new pair of palindromes $p' \in P_n$ and $q' \in P_{n-d}$ with the same sum $p'+q'=p+q$. A typical pair $(p, q)$ will have $\gtrsim C^{-d} n$ disjoint substrings $(a, b)$ and so we can do the swapping in $\gtrsim \exp(C^{-d} n)$ different ways. So a typical sum $p+q \in P_n + P_{n-d}$ has lots of representations and this means that the sumset has to be small. 

Denote $t = [\frac{n}{3(d+2)}]$. For $p = \overline{p_0 p_1 \ldots p_1 p_0} \in P_n$ and $q = \overline{q_0 q_1 \ldots q_1 q_0} \in P_{n-d}$ let $S(p, q)$ denote the number of indices $1 \le j \le t$ such that
\begin{equation}\label{eqp}
\overline{p_{(d+2) j + d+1} p_{(d+2)j+d} \ldots p_{(d+2)j+1} p_{(d+2)j}} = a,    
\end{equation}
\begin{equation}\label{eqq}
\overline{q_{(d+2) j + d+1} q_{(d+2)j+d} \ldots q_{(d+2)j+1} q_{(d+2)j}} = b,    
\end{equation}
i.e. the segments of digits of $p$ and $q$ in the interval $[(d+2)j, (d+2)j +d+1]$ are precisely $a$ and $b$. 

\begin{prop}\label{prop1}
    The number of pairs $(p, q) \in P_n \times P_{n-d}$ such that $S(p, q) \le \frac{t}{2 g^{2d+4} }$ is at most $\exp\left( - \frac{t}{8 g^{2d+4} } \right) |P_n| |P_{n-d}|$.
\end{prop}

\begin{proof}
    Draw $(p, q)$ uniformly at random from $P_n \times P_{n-d}$. Then $S(p, q)$ is a sum of $t$ i.i.d Bernoulli random variables with mean $g^{-2(d+2)}$. So the expectation $\E_{p,q} S(p, q)$ is given by $\mu = t g^{-2(d+2)}$ and by Chernoff bound, 
    $$
    \Pr[S(p, q) \le \mu/2] \le \exp\left(- \mu / 8\right) =  \exp\left( - \frac{t}{8 g^{2d+4} } \right).
    $$
\end{proof}

Now we observe that for any $p=\overline{p_0 p_1 \ldots p_1 p_0} \in P_n$, $q = \overline{q_0 q_1 \ldots q_1 q_0} \in P_{n-d}$, the sum $s=p+q$ has at least $2^{S(p,q)}$ distinct representations $s = p'+q'$ for $(p', q') \in P_n \times P_{n-d}$. Indeed, let $j_1 < \ldots < j_u$ be an arbitrary collection of indices such that (\ref{eqp}) and (\ref{eqq}) hold for $j=j_1, \ldots, j_u$. Let $p'$ and $q'$ be obtained from $p$ and $q$ by replacing the $a$ and $b$-segments on positions $j_1, \ldots, j_u$ by $a'$ and $b'$ and replacing $r(a)$ and $r(b)$-segments on the symmetric positions by $r(a')$ and $r(b')$, respectively. Then we claim that $p' \in P_n$, $q' \in P_{n-d}$ and $p'+q'=p+q$. Indeed, more formally, we can write
$$
p' = p + \sum_{i=1}^u g^{(d+2) j_i} (a' - a) + g^{n-(d+2) j_i - d-1} (r(a') - r(a)),
$$
$$
q' = q + \sum_{i=1}^u g^{(d+2) j_i} (b' - b) + g^{(n-d)-(d+2) j_i - d-1} (r(b') - r(b)),
$$
and so (\ref{eq2}) implies that $p+q=p'+q'$. Since we can choose $j_1 < \ldots < j_u$ to be an arbitrary subset of $S(p,q)$ indices, we get $2^{S(p, q)}$ different representations $p+q=p'+q'$. 

Using this and Proposition \ref{prop1} we get
$$
    |P_n + P_{n-d}| \le \# \left\{p+q ~|~ S(p, q) \ge \frac{t}{2g^{2d+4}}\right\} + \# \left\{p+q ~|~ S(p, q) \le \frac{t}{2g^{2d+4}}\right\} \le 
$$
$$
\le 2^{- \frac{t}{2g^{2d+4}}} |P_n| |P_{n-d}| + \exp\left(- \frac{t}{8g^{2d+4}}\right) |P_n| |P_{n-d}| \le 2 \exp\left(- \frac{n}{30(d+2) g^{2d+4}}\right) |P_n| |P_{n-d}|.
$$
Using this bound we can estimate the part of (\ref{eqsum}) which was not covered by (\ref{eqcrap}):
$$
\sum_{k \ge n \ge m \ge k-\gamma \log k} |P_n + P_m| \le \sum_{k \ge n \ge k-\gamma \log k} \sum_{d = 0}^{\gamma \log k} |P_n + P_{n-d}| \le
$$
$$
\le \sum_{k \ge n \ge k-\gamma \log k} \sum_{d = 0}^{\gamma \log k} 2 \exp\left(- \frac{n}{30(d+2) g^{2d+4}}\right) |P_n| |P_{n-d}| \le \sum_{k\ge n \ge k-\gamma \log k} 2 \exp\left( - \frac{n}{k^{3 \gamma \log g}} \right) g^{k+1} 
$$
so if we take, say, $\gamma = \frac{1}{4 \log g}$ then this expression is less than, say, $k^{-1} g^k \lesssim \frac{X}{\log X}$ provided that $k$ is large enough. Combining this with (\ref{eqcrap}) gives $|(P+P) \cap [X]| \le \frac{X}{(\log X)^{0.1}}$ for large enough $X$ (the proof actually gives $1/4-\varepsilon$ instead of $0.1$ here).
\bibliographystyle{amsplain0.bst}
\bibliography{main}

\end{document}